\documentclass[a4paper]{amsart}
\usepackage[12pt]{extsizes}

\usepackage[english]{babel}
\usepackage{amsfonts,amsmath,amsthm,amssymb,latexsym,mathrsfs,enumitem}
\usepackage{geometry}
\newtheorem{theorem}{Theorem}[section]
\newtheorem{corollary}[theorem]{Corollary}
\newtheorem{lemma}[theorem]{Lemma}
\newtheorem{proposition}[theorem]{Proposition}
\newtheorem{question}[theorem]{Question}

\theoremstyle{plain}

\newtheorem{remark}[theorem]{Remark}

\numberwithin{equation}{section}

\begin{document}

\title{Diagonals of separately continuous maps with values in box products}
\author{Olena Karlova}
\author{Volodymyr Mykhaylyuk}

\maketitle

 \begin{abstract}
   We prove that if $X$ is a paracompact connected space and $Z=\prod_{s\in S}Z_s$ is a product of a family of equiconnected metrizable spaces endowed with the box topology, then for every  Baire-one map  $g:X\to Z$  there exists a separately continuous map  $f:X^2\to Z$ such that $f(x,x)=g(x)$ for all $x\in X$.
 \end{abstract}

\section{Introduction}
Let $X$, $Y$ be topological spaces  and $C(X,Y)=B_0(X,Y)$ be the collection of all continuous maps between $X$ and $Y$.  For $n\geq 1$ we say that a map $f:X\to Y$ belongs to {\it the $n$-th Baire class} if $f$ is a pointwise limit of a sequence of maps $f_k:X\to Y$ from the $(n-1)$-th Baire class. By ${\rm B}_n(X,Y)$ we denote the collection of all maps of the $n$-th Baire class between $X$ and $Y$.

For a map $f:X\times Y\to Z$ and a point $(x,y)\in X\times Y$ we write $f^x(y)=f_y(x)=f(x,y)$. By $CB_n(X\times Y,Z)$ we denote the collection of all mappings $f:X\times Y\to Z$ which are continuous with respect to the first variable and belongs to the $n$-th Baire class with respect to the second one. If $n=0$, then we use the symbol $CC(X\times Y,Z)$ for the class of all separately continuous maps. Now let $CC_0(X\times Y,Z)=CC(X\times Y,Z)$ and for $n\ge 1$ let $CC_n(X\times Y,Z)$ be the class of all maps $f:X\times Y\to Z$ which are pointwise limits of a sequence of maps from $CC_{n-1}(X\times Y,Z)$.

Let $f:X^2\to Y$ be a map. Then the map $g:X\to Y$ defined by $g(x)=f(x,x)$ is called {\it a diagonal of $f$.}

Investigations of diagonals of separately continuous functions $f:X^n\to\mathbb R$ were started in classical works of R.~Baire \cite{Baire}, H.~Lebesgue \cite{Leb1, Leb2} and H. Hahn \cite{Ha} who proved that diagonals of separately continuous functions of $n$ real variables are exactly the functions of the $(n-1)$-th Baire class. On other hand, separately continuous mappings with valued in equiconnected spaces intensively studied starting from \cite{B}.
A brief survey of further developments of these investigations can be found in \cite{KMS1}. At the same time little is known about the possibility of extension of a ${\rm B}_1$-function from the diagonal of $X^2$  to a separately continuous function on the whole $X^2$ when a range space is not metrizable. We know only one paper in this direction~\cite{KMS1} where maps are considered with values in a space $Z$ from a wide class of spaces which contains metrizable equiconnected spaces and strict inductive limits of sequences of closed locally convex metrizable subspaces.

Here we continue investigations in this direction and study maps with values in products of topological spaces endowed with the box topology. The main result of our paper is the following.
\begin{theorem}
  Let $X$ be a paracompact connected space, $Z=\prod_{s\in S}Z_s$ be a product of a family of equiconnected metrizable spaces endowed with the box topology and $g:X\to Z$ be a Baire-one map. Then there exists a separately continuous map  $f:X^2\to Z$ with the diagonal $g$.
\end{theorem}

\section{The case of countable box-products}

We consider a family $(X_s)_{s\in S}$ of topological spaces and put $X_S=\prod_{s\in S}X_s$.
For  $x\in X_S$ and $T\subseteq S$ we define $x|_T$ as the point from $X_T$ with coordinates $(y_s)_{s\in T}$ such that $y_s=x_s$ for all $s\in T$.

The product $X_S$ endowed with the {\it box topology} generated by the family of all boxes $\prod_{s\in S}G_s$, where $(G_s)_{s\in S}$ is a family of open subsets of $X_s$ for every $s\in S$, is called {\it the box product} and is denoted by $\Box_{s\in S} X_s$.

For a fixed point $a=(a_s)_{s\in S}\in X_S$ we consider the set
$$
\sigma(a)=\{(x_s)_{s\in S}: \{s\in S:x_s\ne a_s\} \,\,\mbox{is finite}\}.
$$
If $\sigma(a)$ is endowed with the box topology, then $\sigma(a)$ is called {\it the small box product} of $(X_s)_{s\in S}$ and is denoted by $\boxdot_{s\in S}X_s$ or simply by $\boxdot$ when no confusion will arise.

Further, for all $n\in\omega$ we put
\begin{gather*}
   \boxdot_n=\{(x_s)_{s\in S}:|\{s\in S:x_s\ne a_s\}|=n\},\\
    \boxdot_{\le n}=\bigcup_{k\le n}\boxdot_k
\end{gather*}
and for a finite subset $T\subseteq S$ let
\begin{gather*}
    \boxdot_{T}=\{(x_s)_{s\in S}:x_s=a_s \,\, \Leftrightarrow s\in S\setminus T\}.
\end{gather*}
Obviously,
\begin{gather*}
  \boxdot=\bigsqcup_{n\in\omega}\boxdot_n \,\,\,\mbox{and}\,\,\, \boxdot_n=\bigsqcup_{T\subseteq S, |T|=n} \boxdot_{T}.
\end{gather*}

 The facts below follow  easily from the definition of the box topology and we omit their proof. For another properties of the box topology see \cite{Williams}.
\begin{proposition}\label{properties}
Let $(X_s:s\in S)$ be a family of topological spaces.
\begin{enumerate}
\item Each $\boxdot_{s\in S}X_s$ is a closed subspace of $\Box_{s\in S}X_s$ whenever each $X_s$ is a $T_1$-space.

\item The space $\boxdot_{T}$ is homeomorphic to the finite product $\prod_{t\in T}X_t$ for any finite set $T\subseteq S$.

\item Each $\boxdot_{T}$ is clopen in $\boxdot_n$, where $n=|T|$.

\item Let $X_s$ be a $T_1$-space for every $s\in S$ and $(x_n)_{n\in\omega}$ be a sequence of points from $X=\prod_{s\in S}X_s$. Then $(x_n)_{n\in\omega}$ converges to a point $x\in X$ in the box-topology if and only if $(x_n)_{n\in\omega}$ converges to $x$ in the product topology and there
     exists a number $k\in\omega$ and a finite set $T\subseteq S$ such that $\{x_n:n\ge k\}\subseteq \boxdot_{T}$, where $\boxdot_{T}$ is the subspace of $\boxdot=\sigma(x)$.
\end{enumerate}
\end{proposition}

The last property imply the next fact.
\begin{proposition}\label{prop:cont_with_neigh}
  Let $X$ be a first countable space, $(Y_s:s\in S)$ be a family of $T_1$-spaces and $f:X\to \Box_{s\in S}Y_s$ be a continuous map. Then for every $x\in X$ there exist  an open neighborhood $U$ of $x$ and a finite set $T\subseteq S$ such that $f(z)|_{S\setminus T}=f(x)|_{S\setminus T}$ for all $z\in U$.   In particular, $f(U)\subseteq \sigma(f(x))$.
\end{proposition}

Recall that a topological space $X$ is {\it functionally Hausdorff} if for every  $x,y\in X$ with $x\ne y$, there exists a continuous function $f:X\to [0,1]$ such that $f(x)\ne f(y)$.

\begin{proposition}\label{pr:7.0}
 Let $(X_s:s\in S)$ be a family of Hausdorff spaces and $X\subseteq \Box_{s\in S}X_s$ be a connected subspace.  If
\begin{enumerate}
  \item $X$ is path-connected, or

  \item every $X_s$ is functionally Hausdorff,
\end{enumerate}
 then there exists $x^*\in X$ such that $X\subseteq \sigma(x^*)$.
\end{proposition}

\begin{proof} We fix points $x=(x_s)_{s\in S}$ and $y=(y_s)_{s\in S}$ from $X$ and show that $x$ differs with $y$ for finitely many coordinates.
\begin{enumerate}
\item Let $\varphi:[0,1]\to X$ be a continuous function such that  $\varphi(0)=x$ and $\varphi(1)=y$.
Proposition~\ref{prop:cont_with_neigh} implies that for every $t\in [0,1]$ there exists an open neighborhood $U_t$ of $t$ such that $\varphi(U_t)\subseteq \sigma(\varphi(t))$.  Let points  $t_1,\dots t_n\in [0,1]$ be such that $[0,1]\subseteq \bigcup_{i=1}^{n}U_{t_i}$. Without loss of generality we may assume that the set $\{t_1,\dots t_n\}$ is minimal and  $t_1< t_2<\dots <t_n$. Take   $\tau_1\in U_{t_1}\cap U_{t_2}$,\dots, $\tau_{n-1}\in U_{t_{n-1}}\cap U_{t_n}$ and notice that $\sigma(\varphi(\tau_1))=\dots=\sigma(\varphi(\tau_n))$. Therefore, $\varphi([0,1])\subseteq\sigma(\varphi(\tau_1))$.

\item Assume that the set $\{s\in S: x_s\ne y_s\}$ is infinite and show that there exists a clopen set $U\subseteq \Box_{s\in S}X_s$ such that $x\in U$ and $y\not\in U$. We take a countable subset $$T=\{t_n:n\in \omega\}\subseteq \{s\in S: x_s\ne y_s\},$$ where all $t_n$ are distinct.
    Since each $X_s$ is functionally Hausdorff, for every $s\in T$ one can choose a continuous function $f_s:X_s\to [0,1]$ such that $x_s\subseteq f_s^{-1}(0)$ and $y_s\subseteq f_s^{-1}(1)$. For every $s\in T$ we denote $I_s=[0,1]$ and define a continuous map $f:\Box_{s\in S}X_s\to \Box_{s\in T}I_s$, $f((z_s)_{s\in S})=(f_s(z_s))_{s\in T}$ for all $z=(z_s)_{s\in S}\in \Box_{s\in S}X_s$. Note that the set
\begin{gather*}
V=\{(u_s)_{s\in T}\in[0,1]^T:u_{t_n}\to 0\}
\end{gather*}
is clopen in $\Box_{s\in T}I_s$. It remains to put
$$
U=f^{-1}(V).
$$\end{enumerate}
\end{proof}

Let $X$ be a topological space and $\Delta=\{(x,x):x\in X\}$. A set  $A\subseteq X$ is called {\it equiconnected in $X$} if there exists a continuous mapping  $\lambda:((X\times X)\cup \Delta)\times [0,1]\to X$ such that
$\lambda(A\times A\times [0,1])\subseteq A$, $\lambda(x,y,0)=\lambda(y,x,1)=x$ for all $x,y\in A$ and $\lambda(x,x,t)=x$ for all $x\in X$ and $t\in [0,1]$.
A space is {\it equiconnected} if it is equiconnected in itself.
Notice that any topological vector space is equiconnected, where a mapping  $\lambda$ is defined by $\lambda(x,y,t)=(1-t)x+t y$.

\begin{proposition}\label{pr:7.1}
  Let $(X_s)_{s\in S}$ be a family of equiconnected spaces $(X_s,\lambda_s)$. Then each small box-product $\boxdot_{s\in S}X_s$ is equiconnected.
\end{proposition}

\begin{proof}
For  $z=(z_s)_{s\in S}, w=(w_s)_{s\in S}\in \boxdot_{s\in S}X_s$ and $t\in[0,1]$ we put
\begin{equation}\label{eq:7.1}
 \lambda(z,w,t)=(\lambda_s(z_s,w_s,t))_{s\in S}.
 \end{equation}
It is easy to see that the space $(\boxdot_{s\in S}X_s,\lambda)$ is equiconnected.
\end{proof}

A covering $(X_n:n\in\omega)$ of a topological space $X$ is said to be {\it sequentially absorbing} if for any convergent sequence $(x_n)_{n\in\omega}$ there exists $k\in\omega$ such that $\{x_n:n\in\omega\}\subseteq X_k$.
Let us observe that $(\boxdot_{\le n}:n\in\omega)$ is a sequentially absorbing covering of $\boxdot$ by Proposition~\ref{properties}(4).

A topological space $X$ is said to be {\it strongly $\sigma$-metrizable} if it has a sequentially  absorbing covering (which is called {\it a stratification of $X$}) by metrizable subspaces. A stratification $(X_n)_{n=1}^{\infty}$ of a space $X$ is said to be {\it perfect} if for every  $n\in\mathbb N$ there exists a continuous mapping $\pi_n:X\to X_n$ with $\pi_n(x)=x$ for every $x\in X_n$. Notice that according to \cite{BB} every strongly $\sigma$-metrizable space $X$ is {\it super $\sigma$-metrizable}, that is there exists a covering $(X_n:n\in\omega)$ of $X$ by closed subspaces $X_n$ such that every compact subset of $X$ is contained in some $X_n$. A stratification $(X_n)_{n=1}^{\infty}$ of an equiconnected  strongly $\sigma$-metrizable space $X$ is  {\it compatible with $\lambda$} if $\lambda(X_n\times X_n\times[0,1])\subseteq X_n$ for every $n\in\mathbb N$.

\begin{proposition}\label{pr:7.2}
Let $(X_n)_{n=1}^{\infty}$ be a sequence of metrizable equiconnected spaces $(X_n,\lambda_n)$, $a\in\prod\limits_{n=1}^{\infty}X_n$  and $Z=\sigma(a)=\boxdot_{n\in \omega}X_n$. Then there exists  $\lambda:Z\times Z\times [0,1]\to Z$ such that $(\boxdot_{n\in \omega}X_n,\lambda)$ is strongly $\sigma$-metrizable equiconnected space with a perfect stratification $(Z_n)_{n=1}^{\infty}$ assigned with $\lambda$.
\end{proposition}

\begin{proof} For any  $z=(z_n)_{n=1}^{\infty}$, $w=(w_n)_{n=1}^{\infty}\in Z$ and $t\in[0,1]$ we put
$$
\lambda(z,w,t)=(\lambda_n(z_n,w_n,t))_{n=1}^{\infty}
$$
and notice that the space  $(\boxdot_{n\in \omega}X_n,\lambda)$ is equiconnected.

For every $n\in\mathbb N$ let
$$
Z_n=\{(z_k)_{k=1}^{\infty}\in Z:(z_k=a_k)(\forall k>n)\}.
$$
 The space $\boxdot_{n\in \omega}X_n$ is strongly $\sigma$-metrizable with the stratification $(Z_n)_{n=1}^{\infty}$. Since $\lambda(Z_n\times Z_n\times[0,1])\subseteq Z_n$, the stratification $(Z_n)_{n=1}^{\infty}$ is assigned with $\lambda$. Moreover, for every $n\in\mathbb N$ the map $\pi_n:Z\to Z_n$ is continuous, where $\pi_n((z_k)_{k=1}^\infty)=(w_k)_{k=1}^\infty$ and $w_k=z_k$ for $k\leq n$, $w_k=a_k$ for $k>n$. Hence, the stratification  $(Z_n)_{n=1}^{\infty}$ is perfect.
\end{proof}

Theorem 6 from \cite{KMS1} implies the following result.

\begin{theorem}\label{cor:7.1}
Let $X$ be a topological space, $S$ be a countable set, $(Z_s)_{s\in S}$ be a sequence of metrizable equiconnected spaces, $n\in\mathbb N$ and $g\in {\rm B}_{n-1}(X,\boxdot_{s\in S}Z_s)$. Then there are a separately continuous map $f:X^n\to \boxdot_{s\in S}Z_s$   and a map $h\in CB_{n-1}(X\times X,\boxdot_{s\in S}Z_s)\cap CC_{n-1}(X\times X,\boxdot_{s\in S}Z_s)$ both with the diagonal $g$.
\end{theorem}

\section{The case of uncountable box-products}

\begin{theorem}\label{cor:7.2}
  Let $X$ be a Lindel\"{o}f first countable space, $(Z_s)_{s\in S}$ be a family of metrizable equiconnected spaces  $(Z_s,\lambda_s)$, $n\in\mathbb N$ and $g\in {\rm B}_{n-1}(X,\boxdot_{s\in S}Z_s)$. Then there are a separately continuous map $f:X^n\to \boxdot_{s\in S}Z_s$  and a map $h\in CB_{n-1}(X\times X,\boxdot_{s\in S}Z_s)\cap CC_{n-1}(X\times X,\boxdot_{s\in S}Z_s)$ both with the diagonal $g$.
\end{theorem}

\begin{proof}  Inductively for $m\in\{0,\dots, n-1\}$ we choose families  $(g_\alpha:\alpha\in\mathbb N^m)$ of maps $g_\alpha\in B_{n-m-1}(X,\boxdot_{s\in S}Z_s)$ such that
\begin{equation}\label{eq:7.0}
g_\alpha(x)=\lim\limits_{k\to\infty}g_{\alpha,k}(x)
\end{equation}
for all $x\in X$, $0\leq m\leq n-2$ and $\alpha\in\mathbb N^m$, where $g_\alpha=g$ for a single element $\alpha\in \mathbb N^0$.

 Fix $\alpha\in\mathbb N^{n-1}$. For every $x\in X$  we apply Proposition~\ref{prop:cont_with_neigh} to the continuous map $g_\alpha\in B_0(X,\boxdot_{s\in S}Z_s)$ and take an open neighborhood $U(\alpha, x)$ of $x$ and a finite set $S(\alpha, x)\subseteq S$ such that
 $$
 g_\alpha(z)|_{S\setminus S({\alpha,x})}=a|_{S\setminus S({\alpha,x})}
 $$ for all $z\in U(\alpha,x)$.
 Since $X$ is Lindel\"{o}f,  we choose a countable set $A_\alpha\subseteq X$ such that î $X\subseteq \bigcup\limits_{x\in A_\alpha}U(\alpha,x)$.
 Consider the countable set $$S_0=\bigcup\limits_{\alpha\in \mathbb N^{n-1} }\bigcup\limits_{x\in A_\alpha}S(\alpha,x).$$ Notice that $g_{\alpha}(x)|_{S\setminus S_0}=a|_{S\setminus S_0}$ for all $x\in X$. Then $$g(x)|_{S\setminus S_0}=a|_{S\setminus S_0}$$ for all $x\in X$ according to (\ref{eq:7.0}).
 It remains to apply Theorem~\ref{cor:7.1} for $S=S_0$.
\end{proof}

 \begin{lemma}\label{l:7.4}
Let $X$  be a collectionwise normal space, $(F_n)_{n\in\omega}$ be an increasing sequence of closed subsets of $X$, $A=\bigcup_{n\in\omega}F_n$,
 $(C_i:i\in I)$ be a partition of $A$ by its clopen subsets. Then there exists a sequence $(\mathscr U_n)_{n\in\omega}$ of discrete families $\mathscr U_n=(U_{n,i}:i\in I)$ of functionally open subsets of $X$ such that  $F_n\cap C_i\subseteq U_{n,i}$ for all $n\in\omega$ and $i\in I$.
\end{lemma}

\begin{proof} Fix $n\in\omega$ and put $F_{n,i}=F_n\cap C_i$ for all $i\in I$. Notice that $(F_{n,i}:i\in I)$ is a discrete family of  closed subsets of $X$. Therefore, there exists a discrete family $(V_{n,i}:i\in I)$ of open subsets of $X$ such that $F_{n,i}\subseteq V_{n,i}$ for all $i\in I$. Since $X$ is normal, for every $i\in I$ there exists  a functionally open set $U_{n,i}\subseteq X$ such that $F_{n,i}\subseteq U_{n,i}\subseteq V_{n,i}$.
\end{proof}

\begin{proposition}\label{pr:7.3}
Let  $X$ be a paracompact space, $(Z_s)_{s\in S}$ be a family of equiconnected metrizable spaces $(Z_s,\lambda_s)$, $a\in Z_S$ and $g\in {\rm B}_1(X,\boxdot_{s\in S}Z_s)$.  Then there exist  a sequence of continuous maps $g_n:X\to \boxdot_{s\in S}Z_s$ and a sequence of functionally open sets $W_n\subseteq X^2$ such that
\begin{enumerate}
\item $\{(x,x):x\in X\}\subseteq W_n$ for all $n\in\omega$;

\item $\lim\limits_{n\to\infty}g_n(x_n)=g(x)$ for every  $x\in X$ and for any sequence  $(x_n)_{n\in\omega}$ of points $x_n\in X$ satisfying $(x_n,x)\in W_{n}$ for all $n\in\omega$.
\end{enumerate}
\end{proposition}

\begin{proof} Let $\boxdot=\boxdot_{s\in S}Z_s$ and $(h_n)_{n\in\omega}$ be a sequence of continuous maps $h_n:X\to \boxdot$ which converges to $g$ pointwisely on $X$. For every $s\in S$ we fix a metric $|\cdot - \cdot|_s$ on the space $Z_s$ which generates its topology. For a finite set $R\subseteq S$ and points $z=(z_s)_{s\in S}, w=(w_s)_{s\in S}\in \boxdot$ we put $$|z-w|_R=\max_{s\in R}|z_s-w_s|_s.$$ Moreover, for any $r\in S$  we define the function $\pi_r:\boxdot\to Z_r$, $\pi_r((z_s)_{s\in S})=z_r$; and for any finite set $R\subseteq S$ we define the function $\pi_R:\boxdot\to \boxdot$ in such a way:
$\pi_R((z_s)_{s\in S})=(w_s)_{s\in S}$, where $w_s=z_s$ for $s\in R$ and $w_s=a_s$ for $s\in S\setminus R$.

We show that there exists a partition  $(A_k:k\in\omega)$  of $X$ by functionally $F_\sigma$ sets $A_k$ such that
\begin{gather}\label{conditionA}
\forall k\in\omega\,\,\, \exists n_k, m_k\in\omega : |\{s\in S:\,(\exists n\geq n_k)\,(\pi_s(h_n(x))\ne a_s)\}|=m_k\,\,\, \forall x\in A_k.
\end{gather}
For every $n\leq k$ we define a continuous function $\varphi_{m,n,k}:X\to [0,1]$ by the rule
$$
 \varphi_{m,n,k}(x)=\sup_{|T|\leq m+1}\inf_{s\in T}\max\{|\pi_s(h_n(x))-a_s|_s,\dots ,|\pi_s(h_k(x))-a_s|_s\}.
 $$
Notice that $\varphi_{m,n,k}(x)=0$ if and only if $$|\{s\in S:\,(\exists i\in [n,k])\,(\pi_s(h_i(x))\ne a_s)\}|\leq m.$$ We put
 $$
 X_{m,n}= \bigcap\limits_{k\geq n}\varphi_{m,n,k}^{-1}(0)
 $$
for all $m,n\in\omega$ and notice that the set $X_{m,n}$ is functionally closed in $X$ and $x\in X_{m,n}$ if and only if $$|\{s\in S:\,(\exists i\geq n)\,(\pi_s(h_i(x))\ne a_s)\}|\leq m.$$
Since the sequence $(h_n)_{n=1}^\infty$ converges to $g$ on $X$ pointwisely, $X=\bigcup\limits _{m,n\in\omega}X_{m,n}$ by Proposition \ref{properties}~(4).

Let $\phi:\omega\to \omega^2$ be a bijection such that $\phi^{-1}(m_1,n_1)<\phi^{-1}(m_2,n_2)$ if $m_1+n_1<m_2+n_2$. Now we put
$$
A_k=X_{\phi(k)}\setminus \left(\bigcup_{i<k} X_{\phi(i)}\right)
$$
for every $k\in\omega$. Then every set $A_k$ is functionally $F_\sigma$ as a difference of functionally closed sets and
$$
|\{s\in S:\,(\exists n\geq n_k)\,(\pi_s(h_n(x))\ne a_s)\}|=m_k
$$
for every $x\in A_k$, where $(m_k,n_k)=\phi(k)$.

For every $k\in\omega$ we take an increasing sequence $(F_{k,m})_{m\in \omega}$ of functionally closed subsets of $X$ such that $A_k=\bigcup_{m\in\omega} F_{k,m}$. Moreover, for every $m\in\omega$  we choose a family $(G_{k,m}:0\leq k\leq m)$ of functionally open sets such that $F_{k,m}\subseteq G_{k,m}$ for all $0\leq k\leq m$ and $G_{i,m}\cap G_{j,m}=\emptyset$ for all $0\leq i<j\leq m$.

For every $k\in\omega$ and for any set $R\subseteq S$ with $|R|=m_k$ we put
\begin{gather*}
V_{k,R}=\{x\in A_k: \{s\in S:\,(\exists n\geq n_k)\,(\pi_s(h_n(x))\ne a_s)\}=R\}
\end{gather*}
and show that $V_{k,R}$ is clopen in $A_k$. Let $x'\in A_k\setminus V_{k,R}$ and
$$
R^\prime=\{s\in S:\,(\exists n\geq n_k)\,(\pi_s(h_n(x'))\ne a_s)\}.
$$ Since $x'\in A_k$, $|R^\prime|=m_k=|R|$. On the other hand, $x'\not\in V_{k,R}$. Therefore, $R^\prime\ne R$ and there exists $s\in R^\prime\setminus R$. We choose $n\geq n_k$ such that $\pi_s(h_n(x'))\ne a_s$. Since $h_n$ is continuous, the set
$$
U^\prime=\{x\in X:\pi_s(h_n(x))\ne a_s\}
$$
is an open neighborhood of $x'$ in $X$. Moreover, $U^\prime\cap V_{k,R}=\emptyset$. Thus, $V_{k,R}$ is closed in $A_k$.

Now let $x_0\in V_{k,R}$. For every $r\in R$ we choose a function $f_r\in\{h_n:n\geq n_k\}$ such that $\pi_r(f_r(x_0))\ne a_r$. Since all functions $f_r$ are continuous, there exists a neighborhood $U_0$ of $x_0$ in $X$ such that $\pi_r(f_r(x))\ne a_r$ for every $r\in R$ and $x\in U_0$. Therefore, $U_0\cap A_k\subseteq V_{k,R}$ and $V_{k,R}$ is open in $A_k$.

Moreover, condition~(\ref{conditionA}) implies that
$$
\bigsqcup_{|R|=m_k}V_{k,R}=A_k.
$$
Since every paracompact space is collectionwise normal, we apply Lemma~\ref{l:7.4} and find for every $n\geq k$ a discrete family  $\mathscr U_{k,n}=(U_{k,n,R}:|R|=m_k)$ of functionally open subsets $U_{k,n,R}\subseteq G_{k,n}$ of $X$ such that
$$
C_{k,n,R}=F_{k,n}\cap V_{k,R}\subseteq U_{k,n,R}
$$
for all $R\subseteq S$ with $|R|=m_k$. Let us observe that the family
$$
\mathscr U_{n}=\bigsqcup_{0\leq k\leq n}\mathscr U_{k,n}= (U_{k,n,R}:0\leq k\leq n\,,|R|=m_k)
$$
is discrete in $X$ for all  $n\in\omega$. Let $\varphi_{k,n,R}:X\to [0,1]$ be a continuous function such that $C_{k,n,R}=\varphi^{-1}_{k,n,R}(0)$ and $X\setminus U_{k,n,R}=\varphi^{-1}_{k,n,R}(1)$, $k,n\in\omega$ and $R\subseteq S$ with $|R|=m_k$.

Now for every $n\in\omega$  we define a continuous map $g_n:X\to Z$,
$$
g_n(x)=\left\{\begin{array}{ll}
                         \lambda(\pi_R(h_n(x)),a,\varphi_{k,n,R}(x)), & 0\leq k\leq n, |R|=m_k, x\in U_{k,n,R} \\
                         a, & x\in X\setminus (\bigcup_{U\in \mathscr U_n}U),
                       \end{array}
 \right.
$$
where the function  $\lambda$ is defined by  (\ref{eq:7.1}).

Let us construct a sequence $(W_n:n\in\omega)$ of functionally open sets in $X^2$. For every $n\in\omega$ we consider a functionally closed set
$$
C_n=\bigsqcup_{0\leq k\leq n,\,|R|=m_k}C_{k,n,R}.
 $$
 For every  $x\in C_n$ we choose  $k\leq n$ and $R\subseteq S$ with $|R|=m_k$  such that  $x\in C_{k,n,R}$. Since the map $g_n$ is continuous, we can take a functionally open neighborhood $W_n(x)\subseteq U_{k,n,R}$ of  $x$ such that $|g_n(x')-g_n(x'')|_R\leq \frac{1}{n}$ for any $x',x''\in W_n(x)$. For every $x\in X\setminus C_n$ we put  $W_n(x)=X\setminus C_n$. Since $X$ is paracompact, there exists a locally finite refinement $(O_{\gamma,n}:\gamma\in \Gamma_n)$ of  $(W_n(x):x\in X)$ such that each $O_{\gamma,n}$ is functionally open (see \cite[Theorem 5.1.9]{Eng}). Not we put
 $$
 W_n=\bigcup_{\gamma\in \Gamma_n}O_{\gamma,n}\times O_{\gamma,n}
 $$
 and notice that $W_n$ is functionally open subset of $X^2$ as a locally finite union of functionally open sets.

Clearly,  $(W_n)_{n\in\omega}$ satisfies condition 1) of the Proposition.

We check condition 2). Fix  $x\in X$ and a sequence $(x_n)_{n\in\omega}$ of points  $x_n\in X$ such that $(x_n,x)\in W_{n}$ for all $n\in\omega$. Take $k\in\omega$ such that  $x\in A_k$. Let $R\subseteq S$ be a finite set with $|R|=m_k$ and $i\geq \max\{k,n_k\}$ be a number such that $x\in F_{k,i}\cap V_{k,R}=C_{k,i,R}$. Notice that  $x\in C_{k,j,R}$ for all  $j\geq i$. In particular, $x\in C_j$.

Let $n\geq i$. Since $(x_n,x)\in W_n$, there are $\gamma\in \Gamma_n$ and $y\in X$ such that $x,x_n\in O_{\gamma,n}\subseteq W_n(y)$. Notice that $y\in C_n$, because  $x\in C_n$. We choose $k'\leq n$ and $R'\subseteq S$ such that  $|R'|=m_{k'}$ and  $y\in C_{k',n,R'}$. Since $x\in W_n(y)\subseteq U_{k',n, R'}$, $x\in U_{k,n, R}$ and $\mathscr U_n$ is discrete, we have $k'=k$ and $R'=R$. Hence, $x,x_n\in W_n(y)\subseteq U_{k,n, R}$ and $|g_n(x)-g_n(x_n)|_R\leq \frac{1}{n}$.

Since $n\geq n_k$, condition~(\ref{conditionA}) implies that $\pi_R(h_n(x))=h_n(x)$. Since $x\in C_{k,n,R}$, $\varphi_{k,n,R}(x)=0$. It follows from the definition of $g_n$  that  $g_n(x)=\pi_R(h_n(x))=h_n(x)$. Moreover, since $x_n\in W_n(y)\subseteq U_{k,n,R}$, $\pi_R(g_n(x_n))=g_n(x_n)$.

Let $g(x)=(z_s)_{s\in S}$,  $g_n(x)=(z_{n,s})_{s\in S}$ and $g_n(x_n)=(w_{n,s})_{s\in S}$ for all $n\in\mathbb N$. If  $s\in S\setminus R$, then $w_{n,s}=z_{n,s}=a_s$ for all $n\geq i$. If  $s\in R$, then $|z_{n,s}-w_{n,s}|_s\leq \frac{1}{n}$ for all $n\geq i$. Therefore,
$$
\lim\limits_{n\to\infty}g_n(x_n)=\lim\limits_{n\to\infty}g_n(x)=\lim\limits_{n\to\infty}h_n(x)=g(x),
$$
which completes the proof.
\end{proof}

Now we need the following general construction of separately continuous maps with the given diagonal from \cite{MSF}.

\begin{theorem}\label{th:1.1}
Let  $X$ be a topological space, $Z$ be a Hausdorff space, $(Z_1,\lambda)$ be an equiconnected subspace of $Z$, $g:X\to Z$, $(G_n)_{n=0}^{\infty}$ and $(F_n)_{n=0}^{\infty}$ be sequences of functionally open sets $G_n$ and functionally closed sets $F_n$ in $X^2$,  let  $(\varphi_n)_{n=1}^{\infty}$ be a sequence of separately continuous functions $\varphi_n:X^2\to [0,1]$,  $(g_n)^{\infty}_{n=1}$ be a sequence of continuous mappings $g_n:X\to Z_1$ satisfying the conditions
\begin{enumerate}
  \item[1)] $G_0=F_0=X^2$ and $\Delta=\{(x,x):x\in X\}\subseteq G_{n+1}\subseteq F_n\subseteq G_n$ for every $n\in\mathbb N$;

  \item[2)] $X^2\setminus G_n\subseteq\varphi_n^{-1}(0)$ and $F_n\subseteq \varphi_n^{-1}(1)$ for every $n\in\mathbb N$;

  \item[3)] $\lim\limits_{n\to\infty}\lambda(g_n(x_n),g_{n+1}(x_n),t_n)=g(x)$ for arbitrary $x\in X$, any sequence $(x_n)^{\infty}_{n=1}$ of points $x_n\in X$ with $(x_n,x)\in F_{n-1}$ for all $n\in\mathbb N$, and any sequence  $(t_n)_{n=1}^{\infty}$ of points $t_n\in[0,1]$.
\end{enumerate}
Then the mapping $f:X^2\to Z$,
 \begin{equation*}
 f(x,y)=\left\{\begin{array}{ll}
                         \lambda(g_n(x),g_{n+1}(x),\varphi_n(x,y)), & (x,y)\in F_{n-1}\setminus F_n\\
                         g(x), & (x,y)\in E=\bigcap\limits_{n=1}^{\infty} G_n
                       \end{array}
 \right.
 \end{equation*}
is separately continuous.
\end{theorem}

\begin{theorem}\label{th:7.4}
  Let $X$ be a paracompact space, $(Z_s)_{s\in S}$ be a family of equiconnected metrizable spaces  $(Z_s,\lambda_s)$, $a\in Z_S$ and $g\in {\rm B}_1(X,\boxdot_{s\in S}Z_s)$. Then there exists a separately continuous map  $f:X^2\to \boxdot_{s\in S}Z_s$ with the diagonal $g$.
\end{theorem}

\begin{proof} We use Proposition~\ref{pr:7.3} and choose a sequence $(g_n)_{n\in\omega}$ of continuous maps $g_n:X\to \boxdot_{s\in S}Z_s$ and a sequence $(W_n)^{\infty}_{n=1}$ of functionally open subsets of $X^2$  which satisfy conditions 1) and 2) of Proposition~\ref{pr:7.3}. Let $G_0=F_0=X^2$. Paracompactness of $X$ implies that we can choose sequences  $(G_n)_{n=1}^{\infty}$ and $(F_n)_{n=1}^{\infty}$ of functionally open and functionally closed sets such that
$$
\{(x,x):x\in X\}\subseteq G_{n+1}\subseteq F_n\subseteq G_n\subseteq \bigcap\limits_{k=1}^{n+2}W_k
$$
for every $n\in\omega$. Now we take a sequence $(\varphi_n)_{n=1}^{\infty}$ of continuous functions $\varphi_n:X^2\to [0,1]$ with $X^2\setminus G_n=\varphi_n^{-1}(0)$ and $F_n= \varphi_n^{-1}(1)$ for every $n\in\mathbb N$. It remains to apply Theorem~\ref{th:1.1}.
\end{proof}

A topological space $X$ is {\it strongly countably dimensional} if there exists a sequence $(X_n)_{n=1}^{\infty}$ of sets $X_n\subseteq X$ such that $X=\bigcup_{n=1}^{\infty}X_n$ and ${\rm dim} X_n< n$ for every $n\in\mathbb N$, where by ${\rm dim\,}Y$ we denote the \v{C}ech-Lebesgue dimension of $Y$.

\begin{corollary}\label{cor:7.5}
Let  $X$ be a connected strongly countably dimensional metrizable space, $(Z_s)_{s\in S}$ be a family of metrizable equiconnected spaces $(Z_s,\lambda_s)$ and $g:X\to \Box_{s\in S}Z_s$.
Then the following conditions are equivalent:
\begin{enumerate}
\item[(i)] $g\in B_{1}(X,Z)$;

\item[(ii)] there exists a separately continuous map $f:X^2\to Z$ with the diagonal $g$.
\end{enumerate}
\end{corollary}

\begin{proof}
$(i)\Rightarrow (ii)$. Let $X\ne\emptyset$ and $(g_n)_{n=1}^{\infty}$ be a sequence of continuous maps $g_n:X\to \Box_{s\in S}Z_s$ which converges to $g$ pointwisely on $X$. According to Proposition \ref{pr:7.0}, for every $n\in\mathbb N$ there exists $a_n\in Z_S$ such that $g_n(X)\subseteq \sigma(a_n)$. We fix a point $x_0\in X$ and put $a=g(x_0)$. Since $\lim\limits_{n\to\infty}g_n(x_0)=g(x_0)$, there exists $n_0\in\mathbb N$ such that $g_n(x_0)\subseteq \sigma(a)$ for every $n\geq n_0$. Thus, $\sigma(a_n)=\sigma(a)$ for every $n\geq n_0$. Therefore, $g(X)\subseteq \sigma(a)$ and $g\in {\rm B_1}(X,\boxdot_{s\in S}Z_s)$. It remains to use Theorem~\ref{th:7.4}.

$(ii)\Rightarrow (i)$. It easy to see that according to Proposition \ref{pr:7.0}, for every separately continuous map  $f:X^2\to \Box_{s\in S}Z_s$ there exists $a\in Z_S$ such that $f(X^2)\subseteq \sigma(a)$. By Corollary 5.4 from \cite{KMM}  $f$ is a Baire-one map and, therefore, so is $g$.
\end{proof}

\begin{remark}
 {\rm
 \begin{enumerate}
   \item We use only paracompactness and connectedness of $X$ in the proof of implication $(i)\Rightarrow (ii)$.

   \item Properties of range spaces in all previous results concerning the construction of a separately continuous map with the given Baire-one diagonal allowed us to reduce the case of topological domain space $X$ to a metrizable one.

    \end{enumerate} }
\end{remark}

The last remark implies the following question.
\begin{question}
 Let $X$ be a topological space, $(Z_s)_{s\in S}$ be a family of equiconnected metrizable spaces  $(Z_s,\lambda_s)$, $a\in Z_S$ and $g\in {\rm B}_1(X,\boxdot_{s\in S}Z_s)$. Does there exist a separately continuous map  $f:X^2\to \boxdot_{s\in S}Z_s$ with the diagonal $g$?
\end{question}

\end{document}